\documentclass{amsart}
\usepackage{amsmath}
\usepackage{amsfonts}
\usepackage{amssymb}
\usepackage{graphicx}%
\newtheorem{mytheo}{Theorem}[section]
\newtheorem{mydef}[mytheo]{Definition}
\newtheorem{myrem}[mytheo]{Remark}
\newtheorem{mylem}[mytheo]{Lemma}
\newtheorem{mycoro}[mytheo]{Corollary}
\newtheorem{mypropo}[mytheo]{Proposition}

\usepackage{color}
\usepackage{fancyhdr}
\usepackage{enumerate}

\theoremstyle{remark}

\numberwithin{equation}{section}

 \allowdisplaybreaks[4]
\begin{document}

\title{THE RANK AND SINGULAR VALUES OF THE INHOMOGENEOUS SUBGAUSSIAN RANDOM MATRICES}

\pagestyle{fancy}

\fancyhf{} 


\fancyhead[CO]{\footnotesize INHOMOGENEOUS SUBGAUSSIAN RANDOM MATRIX }

\fancyhead[CE]{\footnotesize G. DAI, Z. SONG AND H. WANG }

\fancyhead[LE]{\thepage}
\fancyhead[RO]{\thepage}
\renewcommand{\headrulewidth}{0mm}

\author{Guozheng Dai}
\address{School of Mathematical Sciences, Zhejiang University, Hangzhou, 310058,  China.}
\email{guozhengdai1997@gmail.com}

\author{Zeyan Song}
\address{School of Mathematics, Shandong University, Jinan, 250100, China.}
\email{zeyansong8@gmail.com}

\author{Hanchao Wang }
\address{School of Mathematics, Shandong University, Jinan, 250100, China.}
\email{hcwang06@gmail.com}

\date{}


\begin{abstract}
Let $A$ be an $n\times n$ random matrix with mean zero and independent inhomogeneous non-constant sub-Gaussian entries. We get that for any $k<c\sqrt{n}$, the probability of the matrix has a lower rank than $n-k$ that is sub-exponential. Furthermore, we get a deviation inequality for the singular values of $A$. 
This extends earlier results of Rudelson's paper in 2024 by removing the assumption of the identical distribution of the entries across the matrix. Our model covers inhomogeneous matrices, allowing different subgaussian moments for the entries as long as their subgaussian moments have a standard upper bound.
In the past advance, the assumption of i.i.d entries was required due to the lack of least common denominators of the non-i.i.d random matrix. We can overcome this problem using a randomized least common denominator (RLCD) from Livshyts in 2021.
\end{abstract}

\maketitle

\section{Introduction}
Let $A$ be an $n \times n$ random matrix; the classical problem in probability is to estimate the probability that the random matrix $A$ is singular, i.e., $\textsf{P}\left( \det{A}=0 \right)$. In particular, we consider the random matrix $A$ with Rademacher entries (taking values from $ \pm 1$ with probability $1/2$ ). A remarkable work \cite{JAMS_singular} due to Kahn, Koml\'{o}s, and Szemer\'{e}di states that
\begin{align}
	\textsf{P}\left( \mathrm{det}(A) = 0 \right) \le  \left( 0.998 + o(1)\right)^{n}.\nonumber
\end{align}
Subsequently,  a lot of work was done to explore the asymptotically optimal exponent; the probability was bounded by $\left( 3/4+0(1) \right)^{n}$ in Tao and Vu \cite{Tao_RSA, Tao_JAMS} and developed further by Bourgain, Vu, and Wood \cite{bour}, which get a bound of $\left( 2^{-1/2}+o(1) \right)^{n}$. Finally, Tikhomirov \cite{Tikhomirov} made it. In particular, he proved that
\begin{align}
    \textsf{P}\left( \mathrm{det}\left( A \right)=0 \right) =\left( \frac{1}{2}+o\left( 1 \right) \right)^{n}.\nonumber
\end{align}
As a natural extension, studying the smallest singular value has attracted widespread attention. Consider the singular values of the random matrix $A$: $s_{1}(A) \ge \cdots \ge s_{n}(A) \ge 0$. The smallest singular value is defined by 
\begin{align}
    s_{n}(A)=\min_{x \in \mathrm{S}^{n-1}}{\Vert Ax \Vert_{2}}.\nonumber
\end{align}
Spielman and Teng \cite{Sconjecture} conjectured that, when $A$ is an i.i.d. Radechmacher random matrix, then for $\varepsilon\ge 0$
\begin{align}
	\textsf{P}(s_{n}(A)\le \varepsilon n^{-1/2})\le \varepsilon+e^{-cn}.\nonumber
\end{align}
In the past 20 years, much work has been done around this conjecture. Most notably, Rudelson and Vershynin \cite{Rudelsonadvance} showed Spielman-Teng's conjecture up to a constant. In particular, they proved for $\varepsilon \ge 0$
\begin{align}\label{Eq_Rudel_result}
    \textsf{P}(s_{n}(A) \le \varepsilon n^{-1/2}) \le C \varepsilon + e^{-cn}.
\end{align}
Indeed, they gave a more general result, which states that the above deviation inequality is valid when the entries of $A$ are i.i.d. subgaussian random variables (see \eqref{Eq_subgaussian} below for the definition) with mean $0$ and variance $1$.
It should be noted that \eqref{Eq_Rudel_result} with $\varepsilon=0$ yields the result that the invertible probability of $A$ is at most $\exp{(-cn)}$. 
\par 
There is also considerable interest in computing the distribution of $s_{n}(A)$ in a general ensemble. Rebrova and Tikhomirov \cite{Rebrova_IJM} recovered the full strength of \eqref{Eq_Rudel_result}, assuming only i.i.d. mean zero, variance one entries. A recent paper \cite{Livshyts_2} by Livshyts, Tikhomirov, and Vershynin established the deviation inequality of $s_{n}(A)$ when $A$ has independent but not uniquely distributed heavy-tailed entries, which may be the most mild assumption in this context. There are many other interesting works (e.g., \cite{Livshyts_1, Rudelson}) in computing the distribution of $s_{n}(A)$, which we will not cover here.\\
We now note 
\begin{align}
    \textsf{P}\left( \mathrm{rank}(A) \le n-1 \right)=\textsf{P}\left( \mathrm{det}(A)=0 \right) \le e^{-cn}.\nonumber
\end{align}
As introduced above, the bound of $\textsf{P}\left( \mathrm{rank}(A) \le n-1 \right)$ is well explored. One may naturally want to know
the behavior of $\textsf{P}\left( \mathrm{rank}(A) \le n-k \right)$. Let $A$ be an i.i.d. Radechmacher random matrix.
 Kahn, Koml\'{o}s, and Szemer\'{e}di showed that the probability that $A$ has a lower rank than $n-k$ is $O\left( f\left( k \right) \right)^{n}$, where $\displaystyle \lim_{k \to \infty}{f(k)} = 0$. 
Very recently, Rudelson \cite{Rudelsonrank} made a breakthrough in the i.i.d. subgaussian case. In particular, he showed for $k<c\sqrt{n}$,
\begin{align}
    \textsf{P}\left( \mathrm{rank}\left( A \right) \le n-k \right) \le e^{ -ckn }.\nonumber
\end{align}
On this basis, one may want to know whether the matrix rank has similar estimates when the matrix entries are not identically distributed. The first main contribution of this paper lies in deriving inequalities for the rank of the random matrix when the entries of the random matrix follow different distributions. Our first result shows that this is indeed possible.
\par 
Before presenting our result, we first introduce some notations: A random variable $X$ is called subgaussian if 
\begin{align}\label{Eq_subgaussian}
    \textsf{E}\exp{\left( -\left( X/K \right)^{2} \right)} < \infty
\end{align}
for some $K>0$ and denote $\Vert X \Vert_{\psi_{2}}$ by   
\begin{align}
    \Vert X \Vert_{\psi_{2}}=\inf{\left( t > 0 : \textsf{E}\left[ \exp{\left( X/t \right)^{2}} \right]  \le 2 \right)}.\nonumber
\end{align} 
Now, we assume that a random variable $X$ satisfying:
\begin{align}\label{Eq_condition}
\textsf{E}X=0, \ \textsf{E}X^{2}=1,  \ \Vert X\Vert_{\psi_{2}}\le K. 
\end{align}

We now give our first main result.
\begin{mytheo}\label{Theo_main1}
	Let $k,n\in \mathbb{N}$ be numbers such that $k<d_{1.1}\sqrt{n}$ and A be an $n\times n$ matrix with independent entries with satisfying \eqref{Eq_condition}. Then
$$\textsf{P}\left( \mathrm{rank}(A)\le n-k\right) \le \exp{(-c_{\ref{Theo_main1}}kn)}.$$ 
where $d_{1.1}$, $c_{\ref{Theo_main1}}>0$ are some constants depending only on $K$. 
\end{mytheo}
We will use two techniques to obtain this result. On the one hand, we introduce a randomized log-least common denominator (RLCD) to study the rank of random matrices, we use the RLCD to estimate the small ball probability in linear spaces, which overcomes the issues brought about by random variables with different distributions. On the other hand, we use the random rounding method to discretize certain closed sets within linear spaces. During the discretization process, we need to estimate the size of the discretized network. Due to the influence of RLCD, we need to calculate the number of points in the discretized network that are not too far away in ``distance'' from the ``integer lattice''. This is the key to this article and will be introduced in detail in Section 3.
\par 
Based on our study of the rank of matrices, we can naturally investigate the $k$-th singular value of a matrix. This is because if we consider the $k$-th singular value of a matrix to be zero, we can immediately conclude that the rank of the matrix is less than $n-k$. Recently, Nguyen \cite{Nguyen_JFA} considered the distribution $s_{n-k+1}(A)$ where $A$ has i.i.d. standard subgaussian entries and proved that for any $\varepsilon>0$, $\gamma \in (0,1)$, and $k \in (c_{0},n/c_{0})$
\begin{align}
    \textsf{P}\left( s_{n-k+1}(A) \le \varepsilon n^{-1/2} \right) \le (C\varepsilon/k)^{(1-\gamma)k^{2}} + \exp{(-cn)}.\nonumber
\end{align}

Based on the techniques used to prove Theorem \ref{Theo_main1}, we have established the small ball probability inequality for a random matrix's k-th smallest singular value when its entries are not identically distributed. The second main contribution of this paper is that our small-ball probability inequality improves some results of  \cite{Dai_arXiv,Nguyen_JFA}.

\begin{mytheo}\label{Theo_main2}
Let $A$ be an $n \times n$ random matrix with independent entries satisfying \eqref{Eq_condition}. For any fixed $\gamma \in (0,\frac{1}{2})$, we have for $\varepsilon > 0$ and $\log{n} \le k \le d_{\ref{Theo_main2}}\sqrt{n}$
\begin{align}
    \textsf{P}\left( s_{n-k+1}\left( A \right) \le \frac{\varepsilon}{\sqrt{n}} \right) \le \left( \frac{C_{\ref{Theo_main2}}\varepsilon}{k} \right)^{\gamma k^{2}} + e^{-c_{\ref{Theo_main2}}kn}.\nonumber
\end{align}
where $d_{\ref{Theo_main2}},C_{\ref{Theo_main2}},c_{\ref{Theo_main2}}$ are some constants depending only on $\gamma$ and $K$.
\end{mytheo}
\begin{myrem}
    Note that we require $k \ge \log{n}$, and it will be explained that this requirement is necessary in Section 5. Therefore, for the case of $k < \log{n}$, only Theorem \ref{Theo_main1} holds. 
\end{myrem}

\begin{myrem}
  This paper assumes that the second moments of the entries of the random matrix are all 1, which is not a necessary condition. Using the same method, we can obtain the conclusions of this paper under the condition that the second moments of the entries of the random matrix are uniformly bounded.
\end{myrem}

\par 
The rest of this paper is organized as follows. In Section 2, we will give the preliminaries of our paper. In Section 3, we will provide the key method to prove the main theorem, we use the RLCD of Livshyts \cite{Livshyts_2} to analyze the small ball probability and we estimate the size of some net so that we can obtain some property of the subspace in $\mathbb{R}^{n}$. Finally, we complete our proof of Theorem \ref{Theo_main1} and Theorem \ref{Theo_main2} in Sections 4 and 5.
\section{Preliminaries}
\subsection{Notation}
\par 
We denote by $[n]$ the set of natural numbers from $1$ to $n$. Given a vector $x\in \mathbb{R}^{n}$, we denote by $\Vert x\Vert_{2}$ its standard Euclidean norm: $\Vert x\Vert_{2}=\left(\sum_{j\in [n]}{x_{j}^{2}} \right)^{\frac{1}{2}}$, and the supnorm is denoted $\Vert x\Vert_{\infty}=\max_{i}{|x_{i}|}$. The unit sphere of $\mathbb{R}^{n}$ is denoted by $S^{n-1}$. The cardinality of a finite set $\mathrm{I}$ is denoted by $\left| \mathrm{I} \right|$.
\par 
If $V$ is a $m\times l$ matrix, we denote $\mathrm{Row}_{i}\left(V\right)$ its $i$-th row and $\mathrm{Col}_{j}\left(V\right)$ its $j$-th column. Its singular values will be denoted by 
$$s_{1}\left(V\right) \ge s_{2}\left(V\right) \ge \cdots \ge s_{m}\left(V\right) \ge 0.$$
The operator norm of $V$ is defined as 
$$\Vert V\Vert =\max_{x\in S^{n-1}}\Vert Vx\Vert_{2},$$
\\
and the Hilbert-Schmidt norm as 
$$\Vert V\Vert_{\mathrm{HS}}=\left(\sum_{i=1}^{m}\sum_{j=1}^{l}v_{i,j}^{2}\right)^{\frac{1}{2}}.$$ 
Note that $\Vert V\Vert=s_{1}\left(V\right)$ and $\Vert V\Vert_{\mathrm{HS}}=\left(\sum_{j=1}^{m}s_{j}\left(V\right)^{2}\right)^{\frac{1}{2}}$.
\par 
For a random variable $X$ we denote by $\bar{X}$ the symmetrization of $X$ defined as $\overline{X}=X-X'$, where $X'$ is an independent copy of $X$. Note that 
\begin{align}
  \textsf{E}{\left| \overline{X} \right|}^{2}=\mathrm{Var}\left( X \right),
\end{align}
where we defined the variance of a random vector $X$ as the covariance of $X$ with itself, that is, $\mathrm{Var}\left( X \right)=\mathrm{Cov}\left(X, X\right)=\textsf{E}\left| X-\textsf{E}X \right|^{2}$.
\par
We denote by $\mathcal{L}(X,t)$ the L\'{e}vy concetration function of a random vector $X \in \mathbb{R}^{m}$:
\begin{align}
    \mathcal{L}(X,t)=\sup_{y \in \mathbb{R}^{m}}{\textsf{P}\left( \Vert x-y \Vert_{2} \le t \right)}.\nonumber
\end{align}
For $x,y \in \mathbb{R}^{n}$, we denoted by $x\star y$ the Schur product
of $x$ and $y$ defined as $x\star y=\left( x_{1}y_{1},...,x_{n}y_{n} \right)^\top$.
\par 
In the proofs of results in this paper, we define $c$, $c',\dots$ as some fixed constant and define $c\left( u\right)$, $C\left( u\right)$ as a constant related to $u$, they depend only on the parameter $u$. Their value can change from line to line.
\subsection{Decomposition of the sphere}
To divide the subspace, we need the following definition:
\begin{mydef}
Let $\delta,\rho \in \left( 0,1 \right)$, we define the sets of sparse, compressible and incompressible vectors as follows:

\begin{itemize}
   \item  $\mathrm{Sparse}\left( \delta \right)=\left\{ x \in \mathbb{R}^{n} : \left| \mathrm{supp}\left( x \right)\right| \le \delta n \right\};$
    \item $\mathrm{Comp}\left( \delta,\rho \right)=\left\{x \in S^{n-1} : \mathrm{dist}\left( x,\mathrm{Sparse}\left( \delta \right) \right) \le \rho \right\};$
    \item $\mathrm{Incomp}\left( \delta,\rho \right)=S^{n-1}\setminus \mathrm{Comp}\left( \delta,\rho \right)$.
\end{itemize}
 
\end{mydef}
\subsection{Concentration and tensorization}
\par 
First, we assume that the entries of the matrix $A$ are independent and satisfying \eqref{Eq_condition}. Without loss of generality, we may assume that $K\ge 1$.
\par We will introduce a tensorization lemma similar to Lemma 2.2 in \cite{Rudelsonadvance} and Lemma 3.7 in \cite{Rudelsonrank}.
\begin{mylem} \label{lmn2.1}
[Tensorization] Let $X_{1},\dots,X_{n}$ be independent non-negative random variables, and let $M$, $m> 0$ such that $\textsf{P}\left(X_{j}\le s\right)\le \left(Ms \right)^{m}$ for all $s\ge s_{0}$. Then
$$\textsf{P}\left(\sum_{j=1}^{n}X_{j} \le nt\right)\le \left(C_{2.2}Mt\right)^{mn}\quad ~\text{for}~\text{all}~~t\ge s_{0},$$
where $C_{2.2}$ is a constant.
\end{mylem}
Next, we will give some concentration inequality. Specifically, we will introduce the estimation about the operator norm and the Hilbert-Schmidt norm of the random matrix.
\begin{mylem}\label{lm2.2}
[Operator norm] Let $m \le n$, and $Q$ be a $m \times n$ random matrix with centered independent entries $q_{i,j}$ such that $\left| q_{i,j} \right| \le 1$. Then
$$\textsf{P}\left( \Vert Q\Vert \ge C_{\ref{lm2.2}}\sqrt{n}\right)\le \exp{\left( -c_{\ref{lm2.2}}n\right)},$$
where $C_{\ref{lm2.2}}$ and $c_{\ref{lm2.2}}$ are constants.
\end{mylem}


\begin{mylem}\label{lm2.3}[Hilbert-Schmidt norm] Let $m \le n$ and let $A$ be an $m \times n$ matrix whose entries are independent satisfying \eqref{Eq_condition}. We have 
$$\textsf{P}\left( \Vert A\Vert_{\mathrm{HS}} \ge 2C_{\ref{lm2.3}}n\right) \le \exp{\left( -c_{\ref{lm2.3}}n^{2}\right)},$$  
where $C_{\ref{lm2.3}}$ and $c_{\ref{lm2.3}}$ are constants.
\end{mylem} 

Lemma \ref{lm2.2} is from Proposition 2.5 in Rudelson and Vershynin,\cite{Rudelsonadvance}, and Lemma \ref{lm2.3} is  from Lemma 3.6 in Rudelson \cite{Rudelsonrank}. The interested reader is referred to their proof.

\subsection{Randomized least common denominators}
The least common denominator of a vector of $\mathbb{R}^{n}$ was first introduced in Rudelson and  Vershynin \cite{Rudelsonadvance} to be a valuable tool to estimate the small ball probability(The L\'{e}vy function of the inner product of a vector of real value and a random vector with independent entries). To prove the main result, we need to estimate the small ball probability of orthogonal projection, so we need the small ball probability for a linear subspace similar to Section 7 of Rudelson and  Vershynin \cite{RudelsonSmall}. However, we know that the property of identically distributed entries in the random vector is the key to estimating small ball probability. Thus, we need a new concentration inequality to characterize small ball probabilities with different distributions. In the following, we will give the Randomized least common denominators introduced in Livshyts \cite{Livshyts_2} to overcome this problem.
\begin{mydef}
Let $V$ be an $m \times n$ (deterministic) matrix, $\xi=\left(\xi_{1},\dots,\xi_{n}\right)$ be a random vector of real value with independent entries satisfying \eqref{Eq_condition} and let $L>0$, $\alpha \in\left(0,1\right)$. Define the Randomized log-least common denominator(RLCD) of $V$ and $\xi$ by
$$RD_{L,\alpha}^{\xi}\left(V\right) = \inf{\left\{\Vert \theta\Vert_{2}:\theta \in \mathbb{R}^{m},\textsf{E}\mathrm{dist}^{2}\left(V^{T}\theta\star \bar{\xi},\mathbb{Z}^{n}\right) < L^{2}\cdot \log_{+}{\frac{\alpha\Vert V^{T}\theta\Vert_{2}}{L}}\right\}}.$$
If $E \subset \mathbb{R}^{n}$ is a linear subspace, we can adapt this definition to the orthogonal projection $P_{E}$ on $E$ setting
$$RD_{L,\alpha}^{\xi}\left(E\right)= \inf{\left\{\Vert y\Vert_{2}:y \in E,\textsf{E}\mathrm{dist}^{2}\left(y\star \bar{\xi},\mathbb{Z}^{n}\right) < L^{2}\cdot \log_{+}{\frac{\alpha\Vert y\Vert_{2}}{L}}\right\}}.$$
Moreover, set $A$ be an $n\times t$ matrix with columns $A_{1},\dots,A_{t}$,and $A_{j}$ having independent entries satisfying \eqref{Eq_condition}. Then, we define the RLCD of $V$ and $A$ by
$$RD_{L,\alpha}^{A}\left( V\right)=\min_{j \in \left[t\right]}{RD_{L,\alpha}^{A_{j}}\left( V\right)}$$
\end{mydef}
We  give the following result to estimate the small ball probability using RLCD.
\begin{mypropo} \label{pro2.6}
[Small ball probability via RLCD] Consider a real-valued random vector $\xi = \left(\xi_{1},\dots,\xi_{n}\right)$ with independent entries satisfying \eqref{Eq_condition} and $V\in \mathbb{R}^{m\times n}$. Then  exists universal constant $c_{\ref{pro2.6}}>0$, for any $L\ge c_{\ref{pro2.6}}\sqrt{m}$, we have 
\begin{align}
\mathcal{L}\left( V\xi,t\sqrt{m} \right) \le \frac{\left(C_{\ref{pro2.6}}L/\left(\alpha\sqrt{m}\right)\right)^{m}}{\mathrm{det}\left(VV^{T}\right)^{1/2}}\left( t+\frac{\sqrt{m}}{RD_{L,\alpha}^{\xi}\left( V\right)}\right)^{m},\quad t \ge 0.
\end{align} 
where $C_{\ref{pro2.6}}$  is an absolute constant.
\end{mypropo}
\begin{myrem}
    The proof of this lemma follows the argument in the proof of Theorem 7.5 in Rudelson \cite{RudelsonSmall}. It is also similar to Proposition 4.1 in Fernandez \cite{MFV},  we only present a brief proof here.
\end{myrem}
\begin{proof}
By Ess\'{e}en's inequality for the L\'{e}vy concentration function of a general random vector $Y$,
\begin{align}
\mathcal{L}\left(Y,\sqrt{m}\right) \le C^{m}\int_{B\left(0,\sqrt{m}\right)}{\left|\phi_{Y}\left(\theta\right)\right|}\,\mathrm{d}\theta,
\end{align}
where $\phi_{Y}\left(\theta\right)=\textsf{E}\exp{\left(2\pi \mathrm{i} \left\langle \theta,Y\right\rangle\right)}$ is the characteristic function of $Y$ and $B\left(0,\sqrt{m}\right)$ is the ball of radius $\sqrt{m}$ centered at 0.
Set $Y=t^{-1}V\xi$ and assume that $\mathrm{Col}_{k}\left(V\right)=V_{k}$. Then
$$\left\langle \theta,Y\right\rangle=\sum_{k=1}^{n}{t^{-1}\left\langle \theta,V_{k}\right\rangle \xi_{k}},$$
and
$$\phi_{Y}\left(\theta\right)=\prod_{k=1}^{n}{\phi_{k}\left(t^{-1}\left\langle \theta,V_{k}\right\rangle \right)}.$$
Applying Ess\'{e}en's inequality, (2.3) yields that
\begin{align}
\mathcal{L}\left(V\xi,t\sqrt{m}\right)\le C^{m}\int_{B\left(0,\sqrt{m}\right)}{\prod_{k=1}^{n}{\left|\phi_{k}\left(t^{-1}\left\langle \theta,V_{k} \right\rangle \right) \right|}}\,\mathrm{d}\theta.
\end{align}
Note that for any $s\in \mathbb{R}$
$$\left|\phi_{k}\left(s\right)\right|^{2}=\textsf{E}\exp{\left(2\pi \mathrm{i}s\bar{\xi}_{k}\right)}=\textsf{E}\cos{\left(2\pi s\bar{\xi}_{k}\right)}.$$
Then for each $k\le n$, we have 
$$\left|\phi_{k}\left(s\right)\right| \le \exp{\left(-\frac{1}{2} \textsf{E}\left[1-\cos{\left(2\pi s\bar{\xi}_{k}\right)}\right]\right)}$$
by using the inequality $\left|x\right| \le \exp{\left(-\frac{1}{2}\left(1-x^{2}\right)\right)}$ that is valid for any $x\in \mathbb{R}$.

Moreover, for any $s \in \mathbb{R}$ and $\xi_{k}$ satisfy (\ref{Eq_condition}), we have
$$\textsf{E} \left[1-\cos{\left(2 \pi s \bar{\xi}_{k}\right)}\right] \ge \tilde{c} \, \textsf{E}\mathrm{dist}^{2}\left(s\bar{\xi}_{k},\mathbb{Z} \right).$$

Now, let us assume that
$$t\ge t_{0}=\frac{\sqrt{m}}{RD_{L,\alpha}^{\xi}\left(V\right)}.$$

For $\theta \in B\left(0,\sqrt{m}\right)$
$$\Vert \frac{\theta}{t}\Vert_{2} \le RD_{L,\alpha}^{\xi}\left(V\right).$$

Then (2.4) yields
\begin{equation}
\begin{aligned}
\mathcal{L}\left(V\xi,t\sqrt{m} \right)
&\le C^{m}\int_{B\left(0,\sqrt{m}\right)}{\prod_{k=1}^{n}{\left|\phi_{k}\left(t^{-1}\left\langle \theta,V_{k}\right\rangle\right)\right|}}\,\mathrm{d}\theta\\
&\le C^{m}\int_{B\left(0,\sqrt{m}\right)}{\exp{\left(-\frac{\tilde{c}}{2}\textsf{E}\mathrm{dist}^{2}\left(V^{T}\theta\star \bar{\xi}/t,\mathbb{Z}^{n}\right)\right)}}\,\mathrm{d}\theta\\
&\le C^{m}\int_{B\left(0,\sqrt{m}\right)}{\exp{\left(-cL^{2} \log_{+}{\frac{\alpha\Vert V^{T}\theta \Vert_{2}}{Lt}}\right)}}\,\mathrm{d}\theta.\nonumber
\end{aligned}
\end{equation}

Set $z=V^{T}\xi$

Then
$$\mathcal{L}\left(V\xi,t\sqrt{m}\right) \le \frac{\left(CLt/\alpha\right)^{m}}{\mathrm{det}\left( VV^{T}\right)^{1/2}}\int_{\mathbb{R}^{m}}{\exp{\left(-cL^{2}\log_{+}\Vert z\Vert_{2} \right)}}\,\mathrm{d}\theta.$$

Since
$$\int_{\mathbb{R}^{m}}{\exp{\left(-cL^{2}\log_{+}\Vert z\Vert_{2} \right)}}\,\mathrm{d}z \le \left(\frac{C}{\sqrt{m}}\right)^{m},$$
we have
$$\mathcal{L}\left( V\xi,t\sqrt{m}\right) \le \frac{\left(CL/\alpha\sqrt{m}\right)^{m}}{\mathrm{det}\left(VV^{T} \right)^{1/2}}\cdot t^{m}.$$
which completes the proof. 

\end{proof}
Moreover, we have this corollary for orthogonal projection of the linear subspace. 
\begin{mycoro}\label{col2.8}
Consider a random vector with real values $\xi$ as in Proposition \ref{pro2.6}, let $E$ be a subspace of $\mathbb{R}^{n}$ with $\mathrm{dim}E = m$, and let $P_{E}$ denote the orthogonal projection on $E$. Then there exists universal constant $c_{\ref{col2.8}}>0$, for every $L \ge c_{\ref{col2.8}}\sqrt{m}$ we have 
\begin{align}
\mathcal{L}\left(P_{E}\xi,t\sqrt{m}\right) \le \left( \frac{C_{\ref{col2.8}}L}{\left(\alpha\sqrt{m}\right)} \right)^{m} \left( t+\frac{\sqrt{m}}{RD_{L,\alpha}^{\xi}\left( E\right)}\right)^{m},\quad t \ge 0,\nonumber
\end{align}
where $C_{\ref{col2.8}}$ is an absolute constant.
\end{mycoro}
  We will first present some properties of the random variables that satisfy (\ref{Eq_condition}).
 \begin{mylem}\label{lm2.4}
    Let $\xi$ be a random variable satisfying \eqref{Eq_condition}, and $\xi'$ is the independent copy of $\xi$, we have set $\overline{\xi}:= \xi -\xi'$ above, then
    \begin{align}
        \textsf{P}\left( |\overline{\xi}| \ge 1 \right) \ge p_{\ref{lm2.4}},
    \end{align}
    where $p_{\ref{lm2.4}}:=p(K)$ is a constant depending only on $K$.
\end{mylem}
\begin{proof}
    Define $X:=\overline{\xi}^{2}$, we get $\textsf{E}X = 2$.
    Applying the Paley-Zygmund inequality for $X$,
    \begin{align}
        \textsf{P}\left( X \ge 1 \right) \ge \left( 1-\frac{1}{2} \right)^{2}\frac{(\textsf{E}X)^{2}}{\textsf{E}X^{2}} \ge  (\textsf{E}X^{2})^{-1}.\nonumber
    \end{align}
    Note that $\textsf{E}X^{2}=6+2\textsf{E}\xi^{4}$ and 
    \begin{align}
        1+\frac{\textsf{E}\xi^{4}}{2K^{4}} \le \textsf{E}e^{\frac{\xi^{2}}{K^{2}}} \le 2.\nonumber
    \end{align}
    Combining the two inequalities mentioned above, we can derive
    \begin{align}
        \textsf{P}\left( |\xi| \ge 1 \right) \ge (6+4K^{4})^{-1} \ge p.\nonumber
    \end{align}
    This completes the proof of this lemma.
\end{proof}
Later, we will fix $p=p_{\ref{lm2.4}}$, and with the help of the above lemma, we can prove that the RLCD of the incompressible part is relatively large. It is a version of Lemma 3.11 in \cite{Rudelsonrank}.
\begin{mylem}\label{Large RLCD}
Let $\delta$, $\rho \in \left(0,1 \right)$ , let $A$ be an $n \times n$ random matrix like Theorem \ref{Theo_main1} and $U$ be an $n\times l$ matrix(deterministic) such that $U\mathbb{R}^{l} \cap S^{n-1} \in \mathrm{Incomp}\left(\delta ,\rho \right)$. Then exists $h_{\ref{Large RLCD}} = h\left( \delta,\rho,K\right) \in \left( 0,1\right)$, for any $\theta \in \mathbb{R}^{l}$ with $\Vert U\theta \Vert_{2} \le h\sqrt{n}$ and $j \in \left[ n\right]$, satisfies
\begin{align}
\textsf{E}\mathrm{dist}^{2}\left( U\theta\star \overline{\mathrm{Col}_{j}\left(A\right)},\mathbb{Z}^{n}\right) \ge L^{2}\log_{+}{\frac{\alpha \Vert U\theta\Vert_{2}}{L}}.
\end{align}
where $\alpha \le \alpha_{0} = \alpha_{0}\left( \delta,\rho,K\right)$.
\end{mylem}
\begin{proof}
Assume that exists $j \in \left[ n\right]$ 
\begin{align}
\textsf{E}\mathrm{dist}^{2}\left( U\theta\star \overline{\mathrm{Col}_{j}\left( A\right)},\mathbb{Z}^{n} \right) < L^{2}\log_{+}{\frac{\alpha \Vert U\theta \Vert_{2}}{L}}.\nonumber
\end{align}
for some $\theta > 0$, we want to prove that $\Vert U\theta\Vert_{2} > h\sqrt{n}$.

Set $\Vert U\theta \Vert_{2} = t $ , $U\theta/t = u$ , $\mathrm{Col}_{j}\left( A\right) = X = \left( X_{1},X_{2},\dots,X_{n}\right)$. Since $\log_{+}{s} < s^{2}$,
we have 
$$\textsf{E}\mathrm{dist}^{2}\left( tu\star \overline{X} , \mathbb{Z}^{n}\right) < \alpha^{2}t^{2}.$$

Let $q \in \mathbb{Z}^{n}$ denote a closest integer vector to $tu\star \overline{X}$, thus
$$\textsf{E}\left| u\star \overline{X}-q/t\right|^{2} < \alpha^{2}.$$

By the Markov's inequality,
\begin{align}
\left| u\star \overline{X}-q/t \right|^{2} < c^{2}\alpha^{2} \nonumber 
\end{align}
with at least $1-\frac{1}{c^{2}}$ probability. Another application of Markov's inequality shows that 
$$\left| u_{i}\overline{X}_{i}-q_{i}/t \right| < \frac{c\alpha}{c_{1}\sqrt{n}}~~\text{for}~~\text{any}~~\ i \in J_{1},$$
where $J_{1}$ is some subset of $[n]$ and $|J_{1}| \ge n-c_{1}n$.

Furthermore, since $\textsf{E}|\overline{X}|^{2} = 2\mathrm{Var}|X| \le 2Kn$,
a similar application of Markov's inequality shows that, with at least $1-\frac{1}{c^{2}}$ probability 
$$|\overline{X}_{i}| \le \frac{c}{c_{2}}\sqrt{2K} ~~\text{for}~~\text{any}~~i \in J_{2},$$
where $J_{2}$ is some subset of $[n]$ and $|J_{2}| \ge n-c_{2}n$.
\par 
Moreover, incompressible vectors are spread which be founded in 
Lemma 3.4 of  Rudelson and Vershynin \cite{Rudelsonadvance}. Thus, 
there exists a set 
$$J_{3} : = \left\{ i : \frac{\rho}{\sqrt{2n}} \le |u_{i}| \le \frac{1}{\sqrt{\delta n}} \right\},$$
satisfies $ |J_{3}| \ge \frac{1}{2}\rho^{2}\delta n$.

Finally, Lemma \ref{lm2.4} shows that $\textsf{P}\left\{ |\overline{X}_{i}| \ge 1 \right\} \ge p$.  It means that there exists $J_{4} \subset J_{3}$,  which satisfies $|J_{4}| \ge \frac{1}{2}|J_{3}|$ with high probability (set it is $p_{b}$, which depending only on $p$)
$$|\overline{X}_{i}| \ge 1 ~~\text{for}~~\text{any}~~ i \in J_{4}.$$

We can choose the constant $c,c_{1},c_{2}$ which depending only on $\delta,\rho,p$.
$$2(1-\frac{1}{c^{2}})+p_{b} > 2,$$
$$n-c_{1}n+n-c_{2}n+\frac{1}{4}\rho^{2}\delta n > 2n .$$
Then, there exists a coordinate $i$ for which we have simultaneously the following three bounds:
$$|u_{i}\overline{X}_{i}-q_{i}/t| \le \frac{c\alpha}{c_{1}\sqrt{n}},\quad 1 \le |\overline{X}_{i}| \le \frac{c}{c_{2}}\sqrt{2K},\quad \frac{\rho}{\sqrt{2n}} \le |u_{i}| \le \frac{1}{\sqrt{\delta n}}.$$

Furthermore, using the triangle inequality, we get 
$$|\frac{q_{i}}{t}| \ge |u_{i}\overline{X}_{i}|-\frac{c\alpha}{c_{1}\sqrt{n}} \ge \frac{\rho}{\sqrt{2n}}-\frac{c\alpha}{c_{1}\sqrt{n}} > 0 .$$
where $0 < \alpha < \alpha_{0} = \frac{c_{1}\rho}{c\sqrt{2}}$.
\par 
Thus $q_{i} \ne 0$, $q_{i} \in \mathbb{Z}$, we necessarily have $|q_{i}|\ge 1$, furthermore
$$|\frac{q_{i}}{t}| \le |u_{i}\overline{X}_{i}| + \frac{1}{\sqrt{\delta n}} < \frac{c\alpha}{c_{1}\sqrt{n}} \cdot \frac{c}{c_{2}}\sqrt{2K} + \frac{1}{\sqrt{\delta n}} < \frac{c}{c_{2}}\frac{\sqrt{K}\rho}{\sqrt{n}} + \frac{1}{\delta n},$$
where $h=\frac{c}{c_{2}}\sqrt{K}\rho + \frac{1}{\sqrt{\delta}} > 0.$
Then $|t| \ge |q_{i}|h\sqrt{n} \ge h\sqrt{n}.$

\end{proof}
\subsection{The number of integer points inside a ball}
We will need to estimate the number of integer points in a ball in $\mathbb{R}^{n}$. The set $B\left(0,R\right)$ is the ball of radius R centered at 0.

\begin{mylem} \label{lemm2.11}
For any $R>0$,
$$\left| \mathbb{Z}^{n} \cap B\left(0,R\right)\right| \le \left(2+\frac{C_{\ref{lemm2.11}}R}{\sqrt{n}}\right)^{n}.$$
where $C_{\ref{lemm2.11}}>0$ is an absolute constant.
\end{mylem}
\subsection{Almost orthogonal systems of vectors}
 In our paper, we need to control the arithmetic structure of the kernel of some random matrix $\mathrm{B}$; This structure is essentially derived from the Randomized least common denominator(RLCD). This forced us to look for a way to divide the subspace. At the same time, we want to find a suitable representative vector in the subspace to estimate the probability that such vectors are in the kernel. We will employ the result of Section 3.1 in Rudelson \cite{Rudelsonrank}, which solves the problem above. 
\begin{mydef}
Let $\nu \in \left( 0,1 \right)$. An l-tuple of vectors $\left( v_{1},v_{2},\dots,v_{l} \right) \subset \mathbb{R}^{n} \setminus \left\{ 0 \right\}$ is called $\nu$-almost orthogonal if the $n \times l$ matirx $V_{0}$ with $\mathrm{Col}_{j}\left( V_{0} \right)=\frac{v_{j}}{\Vert v_{j} \Vert_{2}} $ satisfies:
$$1-\nu \le s_{l}\left( V_{0} \right) \le s_{1}\left( V_{0} \right) \le 1+\nu.$$
\end{mydef}

The following lemma shows how to divide the linear subspace $E$ into the $E\cap W$(where W is a closed set and $W \subset \mathbb{R}^{n} \setminus \left\{0 \right\}$) and the linear subspace $F \subset E$ with high dimension.
This lemma is critical for estimating the arithmetic structure, which is from Lemma 3.3 in Rudelson \cite{Rudelsonrank}.
\begin{mylem} \label{lm2.13}
[Rudelson \cite{Rudelsonrank}]Let $W \subset \mathbb{R}^{n} \setminus \left\{ 0 \right\}$ be the closed set. Let $l<k \le n$, and let $E \subset \mathrm{R}^{n}$ be a linear subspace of dimension $k$. Then, at least one of the following holds:
\begin{enumerate}
    \item There exist vectors $v_{1},\dots,v_{l} \in E\cap W$such that
\begin{itemize}
    \item The l-tuple $\left(v_{1},\dots,v_{l}\right)$is $\left(\frac{1}{8}\right)$-almost orthogonal;
    \item For any $\theta \in \mathbb{R}^{l}$ such that $\Vert \theta \Vert_{2} \le \frac{1}{20\sqrt{l}}$,
$$\sum_{i=1}^{l}\theta_{i}v_{i}\notin W;$$
\end{itemize}
\item There exists a subspace $F \subset E$ of dimension $k-l$ such that $F \cap W=\varnothing$.
\end{enumerate}

\end{mylem}
\begin{myrem}
The assumption $k<c\sqrt{n}$ in Theorem \ref{Theo_main1} is based on condition $\Vert \theta\Vert_{2} \le \frac{1}{20\sqrt{l}}$ in Lemma \ref{lm2.13} (1) and Proposition \ref{Propo_1}(see more details in Section 3).
\end{myrem}
\subsection{Restricted invertibility phenomenon}
In this subsection, we will introduce the method for estimating singular values of deterministic matrices.  The following lemma is from Theorem 6 in Naor and Youssef \cite{Naor_JTDM}.
\begin{mylem} \label{lm2.8}
    [Naor and Youssef \cite{Naor_JTDM}] Assume that M is a full-rank matrix of size $k \times d$ with $k \le d$. Then for $1 \le l \le k-1$, there exists $l$ different indices $i_{1},\dots,i_{l}$ such that the matrix $M_{i_{1},\dots,i_{l}}$ with columns $\mathrm{Col}_{i_{1}}\left( M \right),\dots,\mathrm{Col}_{i_{l}}\left( M \right)$ has the smallest non-zero singular value $s_{l}\left( M_{i_{1},\dots,i_{l}} \right)$ satisfying
    \begin{align}
        s_{l}\left(M_{i_{1},\dots,i_{l}} \right)^{-1} \le C_{\ref{lm2.8}}\min_{r \in \left\{ l+1, \dots , k \right\}}{\sqrt{\frac{dr}{\left( r-l \right) \sum_{i=r}^{k}{s_{i}\left( M \right)^{2}}}}}.\nonumber
    \end{align}
    where $C_{\ref{lm2.8}}$ is an absolute constant.
\end{mylem}
\subsection{Estimating of compressible case}
This section aims to prove that the kernel of $B$ is unlikely to contain an extensive, almost orthogonal system of compressible vectors. The following lemma is from Proposition 4.2 in Rudelson \cite{Rudelsonrank}. 
\begin{mylem} \label{lm2.9}
[ Rudelson \cite{Rudelsonrank}]Let $k,n \in \mathbb{N}$ be such that $k<n/2$ and let $B$ be an $n-k \times n$ matrix whose entries are independent random variables satisfying \eqref{Eq_condition}. There exists $\tau_{\ref{lm2.9}}> 0$ such that the probability that there exists a $\left(\frac{1}{4}\right)$-almost orthogonal l-tuple $x_{1},x_{2},\dots,x_{l} \in \mathrm{Comp}\left( \tau^{2},\tau^{4}\right)$ with $l \le \tau^{3}n$ and 
$$\Vert Bx_{j} \Vert_{2} \le \tau \sqrt{n} ~~\text{for}~~\text{all}~~ j \in [l]$$
is less than $\exp{\left(-c_{\ref{lm2.9}}ln \right)}$, $c_{\ref{lm2.9}}$ is a constant.
\end{mylem}

\begin{myrem} The proof of Proposition 4.2 in Rudelson \cite{Rudelsonrank} is not based on the property of identically distributed variables. Thus, this lemma can be obtained directly. Furthermore, it is worth noting that the proof Proposition 4.2 in Rudelson \cite{Rudelsonrank}  requires that the second moment of the random variables have a given positive lower bound. This is also why the variances of the matrix entries of our setting are assumed to be 1. Naturally, we can also change the variance of the matrix elements from 1 to a variance with a positive lower bound.

\end{myrem}

\section{Incompressible vectors}

Before introducing the main conclusion of this section, we consider the linear subspace spanned by incompressible vectors in $\ker(B)$. We aim to show that vectors with large RLCD are dominant within this linear subspace. Specifically, we hope that the probability of vectors having small RLCD in this subspace is superexponentially small. As seen in the discussion in Section 4 (where the lemma is the primary tool), we need to consider the following event to obtain the above conclusion. An $l$-tuple almost orthogonal system exists in the linear subspace, where all $l$ vectors have small RLCD. The main conclusion of this section states that the probability of this event occurring is superexponentially small.\par 
We now introduce several important constants before presenting our main results.
\begin{align}
L=c\sqrt{k}, \quad 
\alpha = \min{\left\{ \alpha_{0}\left( \tau^{2},\tau^{4},K \right),\alpha_{0}\left( \tau^{2},\tau^{4}/2,K \right) \right\}},\nonumber
\end{align}
and
\begin{align}
r = \min{\left\{ h\left( \tau^{2},\tau^{4},K\right)/2,h\left( \tau^{2},\tau^{4}/2,K \right)/8 \right\}}.\nonumber
\end{align}
where $c$ is absolute constants from Proposition \ref{pro2.6}, $k$ appears in Theorem \ref{Theo_main1}, $p$ and $K$ is a parameter from (\ref{Eq_condition}), and $\tau$ was chosen by Lemma \ref{lm2.9}.
\par 
For ease of writing, we set 
$$\mathrm{d}_{A}\left( x,\mathbb{Z}^{n} \right) = \min_{1 \le i \le n}\sqrt{\textsf{E}\mathrm{dist}^{2}\left( x\star \overline{\mathrm{Col}_{i}\left( A \right)},\mathbb{Z}^{n} \right)}$$
\par 
The following is the main result of this section
\begin{mypropo}\label{Propo_1}
Let $\rho \in \left( 0,\rho_{0} \right)$, where $\rho_{0}=\rho_{0}\left( \tau, K \right)$ is some positive number. Assume that $l \le k \le \frac{\rho}{2}\sqrt{n}$.
Let $B$ be an $(n-k) \times n$ matrix with independent entries satisfying \eqref{Eq_condition}. Consider the event $\mathcal{E}_{\ref{Propo_1}}$ that exist vectors $v_{1},\dots,v_{l} \in \ker\left(B\right)$ having the following properties:
\begin{enumerate}
    \item $2r\sqrt{n} \le \Vert v_{j} \Vert_{2} \le R: = \exp{\left( \frac{\rho^{2} n}{4 L^{2}} \right)}$ for all $j \in [l]$;
    \item $\mathrm{span}\left( v_{1},\dots,v_{l}\right) \cap S^{n-1} \subset \mathrm{Incomp}\left( \tau^{2},\tau^{4}\right)$; 
    \item The vectors $v_{1},\dots,v_{l}$ are $\left( \frac{1}{8}\right)$-almost orthogonal system;
    \item $\mathrm{d}_{A}\left( x,\mathbb{Z}^{n} \right) \le \rho \sqrt{n}$ for all $j \in [l]$;
    \item The $n \times l$ matrix $V$ with columns $v_{1},\dots,v_{l}$ satisfies 
$$\mathrm{d}_{A}\left( v_{j},\mathbb{Z}^{n} \right) > \rho \sqrt{n}$$
for all $\theta \in \mathbb{R}^{l}$ such that $\Vert \theta \Vert_{2} \le \frac{1}{20\sqrt{l}}$ and $\Vert V\theta \Vert_{2} \ge 2r \sqrt{n}$.
\end{enumerate}

Then 
$$\textsf{P}\left( \mathcal{E}_{3.1} \right) \le \exp{\left( -ln \right)}.$$
\end{mypropo}
Let 
\begin{align}
    W=\left\{ (w_{1},\dots,w_{l}) : w_{i} ~~\text{satisfying}~~ (1)-(5) \right\}.\nonumber
\end{align}
Therefore, we can obtain Proposition \ref{Propo_1} through  proving the following inequality:
\begin{align}
    \textsf{P}\left( W \cap \ker{B} \right) \le \exp{(-ln)}.\nonumber
\end{align}
For proving the above inequality, consider a vector $\textsf{d}=\left( d_{1},\dots,d_{l} \right) \in [r\sqrt{n},R]^{l}$, $d_{i}=2^{s_{i}}, s_{i} \in \mathbb{Z}$ and define the $W_{d}$ be the subset of $W$ satisfying for $(v_{1},\dots,v_{l}) \in W_{d}$:
\begin{align}
    \Vert v_{j} \Vert_{2} \in [d_{j},2d_{j}] ~~\text{for}~~\text{all}~~ \ j \in [l].\nonumber
\end{align}
Now, we only need to estimate the intersection of $W_{d}$ and $\ker{B}$, then we divide the proof of Proposition \ref{Propo_1} into three parts. The first part estimates the size of the net that discretizes $W_{d}$, and the second part begins with a detailed introduction to how to discretize $W_{d}$, and finally completes the proof of proposition.

\subsection{Size of the net}
In this section, our main result is that the size of the net is ``super-exponentially'' large. First of all, we set temporarily undetermined constant $\delta$, which will be determined in Section 3.2,  it satisfies that 
\begin{align}
\delta > 0,  \quad \delta \le \rho.
\end{align}
Now, we give the main result of this section, which states that it is close to the ``integer lattice'' for the points in $(\delta \mathbb{Z}^{l})$ are a minority.
\begin{mylem}\label{Lem_1}
Let $d = \left( d_{1},\dots,d_{l} \right)$ be a vector such that $d_{j} \in [r\sqrt{n},R]$, for all $j \in [l]$. 
Let $\delta$ be as in (3.1) and $\mathcal{N}_{d} \subset \left( \delta \mathbb{Z}^{n} \right)^{l}$ be the set of all $l$-tuples of vectors $u_{1},\dots,u_{l}$ such that 
$$\Vert u_{j} \Vert_{2} \in [\frac{1}{2}d_{j},4d_{j}] ~~\text{for}~~\text{all}~~ \ j \in [l],$$
$$\mathrm{d}_{A}\left( u_{j},\mathbb{Z}^{n} \right) < 2 \rho \sqrt{n}$$
and
$$\mathrm{span}\left( u_{1},\dots,u_{l} \right) \cap S^{n-1} \subset \mathrm{Incomp}\left( \tau^{2},\tau^{4}/2 \right).$$
Then 
$$\left| \mathcal{N}_{d} \right| \le \left( \frac{C_{\ref{Lem_1}}\rho^{c_{\ref{Lem_1}}}}{r\delta} \right)^{ln}\left( \prod_{j=1}^{l}{\frac{d_{j}}{\sqrt{n}}} \right)^{n} \cdot n^{l},$$
where $C_{\ref{Lem_1}},c_{\ref{Lem_1}}> 0 $ depending only on $\tau,K$.
\end{mylem}
To prove this lemma, we need to estimate the number of vectors in the set that satisfy the property $\mathrm{d}_{A}\left( x,\mathbb{Z}^{n} \right) < 2\rho \sqrt{n}$, thus we need the following lemmas.
\begin{mylem}\label{Lem_2}
Set $\Lambda_{d_{j}}: = \left\{ \Vert u \Vert_{2} \in [\frac{1}{2}d_{j},4d_{j}]  \right\} \cap \left\{ \frac{u_{j}}{\Vert u_{j} \Vert_{2}} \in \mathrm{Incomp}\left( \tau^{2},\tau^{4}/2 \right) \right\} \cap \delta \mathbb{Z}^{n}$, let W be a vector uniformly distributed in the set $\Lambda_{d_{j}}$. Then for all $\mathrm{Col}_{i}\left( A \right),\ i \in [n]$.
$$\textsf{P}_{W} \left\{ \textsf{E}\mathrm{dist}^{2}\left( W\star \overline{\mathrm{Col}_{i}\left( A \right)},\mathbb{Z}^{n} \right) \le \left( 2\rho \sqrt{n} \right)^{2} \right\} < \left( C_{\ref{Lem_2}}\rho \right)^{c_{\ref{Lem_2}}n}.$$
where $C_{\ref{Lem_2}},c_{\ref{Lem_2}}>0$ depending only on $\tau,K$.
\end{mylem}
For proving Lemma \ref{Lem_2},  we set $$\Lambda_{d_{j}}^{J} = \Lambda_{d_{j}} \cap \left\{ u: \frac{|u_{i}|}{\Vert u \Vert_{2}} \ge \frac{\tau^{4}}{2\sqrt{2n}},~~\text{for}~~\text{all}~~ i \in J \right\}.$$ Note that there exists $c_{\tau}=c\left( \tau \right)$, such that $\Lambda_{d_{j}}^{J} \ne \emptyset$ if  $\left| J \right| \ge c_{\tau}n$, and $\Lambda_{d_{j}}^{J} = \emptyset$, if $\left| J \right| \le c_{\tau}n$.

\begin{mylem} \label{lm3.4}
For any $t \in \left( 0,1\right)$ and $u > 1$, there exists $n_{0}=n\left( t,u,\tau,K \right)$. When $n \ge n_{0}$, for all $\left| J \right| \ge c_{\tau}n$. If $x \in \mathbb{R}^{n}$ satisfy that
\begin{align}
\Vert x \Vert_{2}^{2} \le \frac{c_{\tau}t}{2}u^{2}n \quad , \quad \left|\left\{ i \in J:\left| x_{i} \right| \ge 1 \right\}\right| \ge t\left| J \right|.\nonumber
\end{align}
and if $W$ be a vector uniformly distributed on the set $\Lambda_{d_{j}}^{J}$.
Thus, for $\rho \ge \delta$, we have 
\begin{align}
\textsf{P}_{W}\left\{ \mathrm{dist}\left( W\star x,\mathbb{Z}^{n} \right)^{2} \le \rho^{2}n \right\} \le \left( C_{\ref{lm3.4}}\rho \right)^{c_{\ref{lm3.4}}n}\nonumber
\end{align}
where $C_{\ref{lm3.4}},c_{\ref{lm3.4}}> 0$ depending only on $t,u,\tau,K$.
\end{mylem}
\begin{proof}
\par 
Assume that $I_{0} = \left\{ i \in J:|x_{i}|  \ge 1 \right\}$
and $I = \left\{ i \in I_{0}: |x_{i}| \le u \right\}$, then
$$u^{2}\left( |I_{0}|-|I| \right) \le \sum_{i \in I_{0}\setminus I} {|x_{i}|^{2}} \le \sum_{i \in I_{0}}{|x_{i}|^{2}} \le \frac{c_{\tau}t}{2}u^{2}n.$$
We obtain
$$|I| \ge \frac{c_{\tau}t}{2}u^{2}n.$$

Furthermore, set $W = \delta Y$, on the one hand, for $i \in I$ 
$$\delta |x_{i}| \le u \delta \le u \rho .$$
On the other hand, for $i \in I$
$$|w_{i}||x_{i}| \ge \frac{\tau^{4}}{2\sqrt{2n}}\Vert w \Vert_{2} \ge c\tau^{4}r.$$
Thus,  the random variable $|w_{i}||x_{i}|$ is uniformly distributed on a lattice interval of diameter at least $c\tau^{4}r$,  and 

$$\textsf{P}\left\{ \mathrm{dist}\left( w_{i}x_{i},\mathbb{Z} \right) < \epsilon \right\} \le \frac{c\epsilon}{\tau^{4}r} \quad \text{for}~\text{any}~\epsilon \ge u\rho,$$
where $c>0$ is an absolute constant.
\par 
 Using  Lemma \ref{lmn2.1},  we have that 
$$\textsf{P}\left\{ \mathrm{dist}\left( W\star x,\mathbb{Z}^{n} \right)^{2} < \epsilon^{2}|I| \right\} \le \left( \frac{\tilde{c}\epsilon}{\tau^{4}r} \right)^{|I|} \quad\text{for}~\text{any}~\epsilon \ge u\rho.$$

Let $\epsilon=\max{\left( u\rho,\sqrt{\frac{2}{c_{\tau}t}}\rho \right)}$, we obtain
$$\textsf{P}\left\{ \mathrm{dist}\left( W\star x,\mathbb{Z}^{n} \right)^{2} \le \rho^{2}n \right\} \le \left( C\rho \right)^{cn}.$$

\end{proof}

\begin{proof}[\textsf{Proof of Lemma \ref{Lem_2}}]
Note that $\Lambda_{d_{j}}=\bigcup_{J \subset [n]}{\Lambda_{d_{j}}^{J}}$
where $|J| \ge c_{\tau}n$.
We now only need to prove that, if $W$ is a vector uniformly distributed on the set of $\Lambda_{d_{j}}^{J}$, then 
\begin{align}
\textsf{P}_{W}\left\{ \mathrm{E}\mathrm{dist}^{2}\left( W\star \overline{x},\mathbb{Z}^{n} \right) \le 4\rho^{2}n \right\} \le \left( C\rho \right)^{cn}     \nonumber
\end{align}
where $x=\left( x_{1},\dots,x_{n} \right)$ satisfy that (\ref{Eq_condition}).
\par 
Applying Lemma \ref{lm2.4},  $\textsf{P}\left\{ |\overline{x}_{i}| \ge 1 \right\} \ge p$.
On the one hand, with at least $1-\frac{1-p}{1-t}$ probability, 
\begin{align}
|\left\{ i \in J: |\overline{x}_{i}| \ge 1 \right\}| \ge t|J|.
\end{align}
On the other hand,  if $X_{i}$ is the subgaussian random variables and $\Vert X_{i} \Vert_{\psi_{2}} \le K$, 
$$\textsf{P}\left\{ \Vert \overline{X} \Vert_{2}^{2} \ge Cn \right\} \le \exp{\left( \left(2 \log{2}-\frac{C}{K}\right)n \right)}$$
holds. Thus,  set $C=\frac{c_{\tau}t}{2}u^{2} > 2K\log{2}$ and we  have 
\begin{align}
\Vert x \Vert_{2}^{2} \ge \frac{c_{\tau}t}{2}u^{2}n
\end{align}
with probability at least $1-\exp{\left( \left( 2 \log{2}-\frac{c_{\tau}tu^{2}}{2K} \right)n \right)}$.
\par 
Denote the event $\mathcal{E}_{\ref{Lem_2}}$ is that (3.2) and (3.3) occur at the same time.
Then we can choose $t=t\left( \tau,K \right)$ and $u=u\left( \tau,K,t \right)$ 
satisfies 
$$\textsf{P}\left( \mathcal{E}_{\ref{Lem_2}}\right)>\frac{p}{2}.$$
Applying Lemma \ref{lm3.4} with $2^{3/2}p^{-1/2}\rho$,
\begin{equation}
\begin{aligned}
\textsf{P}_{W}\left\{ \mathrm{E}\mathrm{dist}^{2}\left( W\star \overline{x},\mathbb{Z}^{n} \right) \ge 4\rho^{2}n \right\}
& \ge \textsf{P}_{W}\left\{ \mathrm{E}_{\mathcal{E}_{3.2}}\mathrm{dist}^{2}\left( W\star \overline{x},\mathbb{Z}^{n} \right) \ge 8p^{-1}\rho^{2}n \right\}\\
& \ge 1-\left( C\rho \right)^{cn}.\nonumber
\end{aligned}
\end{equation}
We have 
\begin{align}
\textsf{P}_{W}\left\{ \mathrm{E}\mathrm{dist}^{2}\left( W\star \overline{x},\mathbb{Z}^{n} \right) < 4\rho^{2}n \right\} \le \left( C\rho \right)^{cn},
\end{align}
where $C,c > 0$ depending only on $\tau,K$.
\end{proof}

\begin{proof}[\textsf{Proof of Lemma \ref{Lem_1}}]
Applying Lemma \ref{lemm2.11}
\begin{align}
\left| \Lambda_{d_{j}} \right| \le \left| \delta \mathbb{Z}^{n} \cap B\left( 0,4d_{j} \right) \right| \le \left( 2+\frac{cd_{j}}{\delta \sqrt{n}} \right)^{n} \le \left( \frac{\tilde{c}d_{j}}{r\delta\sqrt{n}} \right)^{n}.\nonumber
\end{align}
Furthermore 
\begin{equation}
\begin{aligned}
\left| \mathcal{N}_{d} \right| 
& \le \left[ \sum_{j=1}^{n}{\textsf{P}_{W}\left( \mathrm{E}\mathrm{dist}^{2}\left( W\star \overline{\mathrm{Col}_{j}\left( A\right)},\mathbb{Z}^{n} \right) < 4\rho^{2}n \right)} \right]^{l} \cdot \prod_{j=1}^{l}|\Lambda_{d_{j}}| \\
& \le n^{l} \cdot \left( \frac{C\rho^{c}}{r\delta} \right)^{ln}\cdot \left( \prod_{j=1}^{l}{\frac{d_{j}}{\sqrt{n}}} \right)^{n}.\nonumber
\end{aligned}
\end{equation}
This completes the proof of Lemma \ref{Lem_1}.
\end{proof}
\subsection{Approximating}
This section is crucial for proving Proposition \ref{Propo_1}, which shows that for any $\left( v_{1},\dots,v_{l} \right) \in \mathrm{W}_{d}$, there exists $\left( u_{1},\dots,u_{l} \right) \in \mathcal{N}_{d}$ which approximates it in various ways. We also need to control the RLCD of the matrix $U$ formed by $u_{1},\dots,u_{l}$. 

The following lemma is about the randomized rounding. To the best of authors' knowledge, due to Livshyts \cite{Livshyts_1}, this method is used to choose a best lattice approximation for a vector.

\begin{mylem} \label{lm3.5}
Let $k \le cn$, and $d=\left( d_{1},\dots,d_{l} \right) \in [r\sqrt{n},R]^{l}$. Let $\delta >0$ be a small enough constant satisfying (3.1). Let $B$ be an $(n-k) \times n$ matrix such that $\Vert B \Vert_{\mathrm{HS}} \le 2kn$. For any $v_{1},\dots,v_{l} \in W_{d} \cap \ker{B}$, there exist $u_{1},\dots,u_{l} \in \mathcal{N}_{d}$ with the following properties:

\begin{enumerate}
    \item $\Vert u_{j}-v_{j} \Vert_{\infty} \le \delta \ for \ all \ j\in [l]$;
\item Let $U$ and $V$ be $n \times l$ matrices with columns $u_{1},\dots,u_{l}$ and $v_{1},\dots,v_{l}$ respectively. Then 
$$\Vert U-V \Vert \le C \delta \sqrt{n};$$
\item $\left( u_{1},\dots,u_{l} \right)$ is $(\frac{1}{4})$-almost orthogonal;
\item $\mathrm{span}\left( u_{1},\dots,u_{l} \right) \cap S^{n-1} \subset \mathrm{Incomp}\left( \tau^{2},\tau^{4}/2 \right)$;
\item $\mathrm{d}_{A}\left( u_{j},\mathbb{Z}^{n} \right)<2\rho\sqrt{n};$
\item $U$ is a matrix as in (2), then
$$\mathrm{d}_{A}\left( U\theta,\mathbb{Z}^{n} \right) > \frac{\rho}{2} \sqrt{n}$$
for any $\theta \in \mathbb{R}$ satisfying 
$$\Vert \theta \Vert_{2} \le \frac{1}{20\sqrt{l}} \quad and \quad \Vert U\theta \Vert_{2} \ge 8r\sqrt{n};$$
\item $\Vert Bu_{j} \Vert_{2} \le 2K\delta n~~\text{for}~~\text{all}\ j \in[l]$. 
\end{enumerate}

\end{mylem}
\begin{proof}
Let $(v_{1},\dots,v_{l}) \in W_{d}$, Choose $\left( \overline{v}_{1},\dots,\overline{v}_{l} \right) \in \delta \mathbb{Z}^{n}$ be such that 
$$v_{j} \in \overline{v}_{j} + \delta [0,1]^{n}.$$
Define independent random variables $\varepsilon_{ij}$, $i \in [n]$, $j \in [l]$ by setting 
$$\textsf{P}\left( \varepsilon_{ij} = \overline{v}_{j}(i)-v_{j}(i) \right) = 1- \frac{v_{j}(i)-\overline{v}_{j}(i)}{\delta}$$
and 
$$\textsf{P}\left( \varepsilon_{ij} = \overline{v}_{j}(i) - v_{j}(i) + \delta \right) = \frac{v_{j}(i) - \overline{v}_{j}(i)}{\delta}.$$
Then we have $|\varepsilon_{ij}| \le \delta$ and $\textsf{E}\varepsilon_{ij} = 0$. Let 
$$u_{j} = v_{j} + \sum_{i=1}^{n}{\varepsilon_{ij}e_{i}} \in \delta \mathbb{Z}^{n}.$$
From the proof of Lemma 5.3 in \cite{Rudelsonrank}, we have (1) and (2) occur with probability at least $1-\exp{(-cn)}$. Furthermore, we can get the (3), (4) by applying (1) and (2).
We will check (5) and (6) follows from (1) and (2).

Note that for any $i \in [n]$.
\begin{equation}
\begin{aligned}
\mathrm{dist}\left( u_{j}\star \overline{\mathrm{Col}_{i}\left(A \right)},\mathbb{Z}^{n} \right) 
& \le \mathrm{dist}\left( v_{j}\star \overline{\mathrm{Col}_{i}\left( A \right)},\mathbb{Z}^{n} \right)+\Vert u_{j}-v_{j} \Vert_{\infty} \Vert \overline{\mathrm{Col}_{i}\left( A \right)} \Vert_{2}\\
& \le \rho \sqrt{n} + \delta \Vert \overline{\mathrm{Col}_{i}\left( A \right)} \Vert_{2}.\nonumber
\end{aligned}
\end{equation}

Then 
\begin{equation}
\begin{aligned}
\mathrm{d}_{A}\left( u_{j},\mathbb{Z}^{n} \right)  
& =\sqrt{\min_{i}{\textsf{E}\mathrm{dist}^{2}\left( u_{j}\star \overline{\mathrm{Col}_{j}\left( A \right)}, \mathbb{Z}^{n} \right)}}\\
& \le \sqrt{\min_{i}\textsf{E}\left( \rho \sqrt{n} + \delta \Vert \overline{\mathrm{Col}_{j}\left( A \right)} \Vert_{2} \right)^{2}}\\
& < 2 \rho \sqrt{n}.\nonumber
\end{aligned}
\end{equation}

Furhermore, since the $(u_{1},\dots,u_{l})$ and $(v_{1},\dots,v_{l})$ are $\left( \frac{1}{2} \right)$-almost orthogonal, and $\Vert v_{j} \Vert_{2} \ge \frac{1}{2} \Vert u_{j} \Vert_{2}$, thus
\begin{align}
\Vert V \theta \Vert_{2}^{2} \ge \frac{1}{4}\sum_{j=1}^{l}{\theta_{j}^{2}\Vert v_{j} \Vert_{2}^{2} } \ge \frac{1}{16}\sum_{j=1}^{l}{\theta_{j}^{2}\Vert u_{j} \Vert_{2}^{2}} \ge \frac{1}{64} \Vert U\theta \Vert_{2}^{2} \ge r^{2}n.\nonumber
\end{align} 

As $(v_{1},\dots,v_{l}) \in W_{d}$, 

we have $\mathrm{d}_{A}\left( V\theta , \mathbb{Z}^{n} \right) \ge \rho \sqrt{n}$. 

Since 
$$\Vert (U-V)\theta \Vert_{\infty} = \max_{i \in [n]}{\left| \sum_{j=1}^{l}{\varepsilon_{ij}\theta_{j}} \right|} 
\le \max_{i \in [n] }{\sqrt{\sum_{j=1}^{l}{\varepsilon_{ij}^{2}}} \sqrt{\sum_{j=1}^{l}{\theta_{j}^{2}}}} \le \delta \sqrt{l} \cdot \frac{1}{20\sqrt{l}},$$

We obtain, for any $i \in [n]$,
\begin{equation}
\begin{aligned}
\mathrm{dist}\left( U\theta\star \overline{\mathrm{Col}_{i}(A)},\mathbb{Z}^{n} \right)
& \ge \mathrm{dist}\left( V\theta\star \overline{\mathrm{Col}_{i}(A)},\mathbb{Z}^{n} \right) - \Vert (U - V) \theta \Vert_{\infty} \Vert \overline{\mathrm{Col}_{i}(A)} \Vert_{2}\\
& \ge \rho \sqrt{n} - \delta \sqrt{l} \cdot \frac{1}{20 \sqrt{l}} \Vert \overline{\mathrm{Col}_{j}(A)} \Vert_{2}\\
& \ge \rho\sqrt{n} - c\delta \Vert \overline{\mathrm{Col}_{j}(A)} \Vert_{2}.\nonumber
\end{aligned}
\end{equation}

Therefore,
\begin{equation}
\begin{aligned}
\mathrm{d}_{A}\left( U\theta,\mathbb{Z}^{n} \right) 
& = \sqrt{\min_{i}{\textsf{E}\mathrm{dist}^{2}\left( U\theta\star \overline{\mathrm{Col}_{i}(A)},\mathbb{Z}^{n} \right) }}\\
& \ge \sqrt{\min_{i}{\textsf{E}\left( \rho \sqrt{n} - c\delta \Vert \overline{\mathrm{Col}_{j}(A)} \Vert_{2} \right)^{2}}} \\
& \ge \frac{\rho}{2} \sqrt{n}.\nonumber
\end{aligned}
\end{equation}

 In fact, we choose the small enough $\delta$ for check (5) and (6).

 Sum up, we have the (1)-(6) occur with the probability at least $1- \exp{\left( -cn \right)}$.

\par  Note that 
$$\textsf{P}\left( \Vert Bu_{j} \Vert_{2} \le 2K\delta n \text{for}~\text{all}~ \ j \in [l] \right) \ge 2^{-l}. $$

Thus 
$$1-\exp{(-cn)} +2^{-l} > 1.$$

It means that there exists $(u_{1},\dots,u_{l}) \in \mathcal{N}_{d}$ satisfying (1)-(7). It completes the proof of this lemma.
\end{proof}
\subsection{The proof of Proposition 3.1}
Before proving  Proposition \ref{Propo_1}, we need the following lemma. Firstly, we fix $\delta$, which has been chosen in Lemma \ref{lm3.4}.
\begin{mylem} \label{lm3.6}
Let $d = (d_{1},\dots,d_{l}) \in [r\sqrt{n},R]^{l}$, let $ck \le l \le k \le \frac{\delta}{20}\sqrt{n}$. Then 
$$\textsf{P}\left( W_{d}\cap \ker{(B)} \ne \emptyset \right) \le \exp{(-2ln)}.$$
\end{mylem}
\begin{proof}
Let $\tilde{\mathcal{N}}_{d} \subset \mathcal{N}_{d}$ be the set of $\left( u_{1},\dots,u_{l}\right) \in \mathcal{N}_{d}$ which satisfy (3)-(7) of Lemma \ref{lm3.5}. Let $U$ be the $n \times l$ matrix with columns 
$u_{1},\dots,u_{l}$.
Firstly, we show that
\begin{align}
\mathrm{RD}_{L,\alpha}^{A}\left( U^{T} \right) \ge \frac{1}{20\sqrt{l}}.
\end{align}
Take $\theta \in \mathbb{R}^{l}$ such that $\Vert \theta \Vert_{2} \le \frac{1}{20\sqrt{l}}$.
Assume that 
$$\Vert U\theta \Vert_{2} \le 8r\sqrt{n}\le h\left( \tau^{2},\tau^{4}/2,K \right)\sqrt{n},$$
applying Lemma \ref{Large RLCD} with $\alpha < \alpha_{0}(\tau^{2},\tau^{4}/2,K)$ yields
$$\mathrm{d}_{A}\left( U\theta,\mathbb{Z}^{n} \right) \ge L\sqrt{\log_{+}\frac{\alpha \Vert U\theta \Vert_{2}}{L}}.$$
Assume that $\Vert U\theta \Vert_{2} \le \sqrt{l} \max_{j \in [l]}{\Vert u_{j} \Vert_{2}} \le \sqrt{l}R$,
hence 
$$L\sqrt{\log_{+}{\frac{\alpha\Vert U\theta \Vert_{2}}{L}}} \le \frac{\rho}{2}\sqrt{n}.$$
By the (6) of Lemma \ref{lm3.5}, we get 
\begin{align}
\mathrm{d}_{A}\left( U\theta,\mathbb{Z}^{n} \right) > \frac{\rho}{2}\sqrt{n},
\end{align} 
whenever $\theta \in \mathbb{R}^{n}$ satisfies 
$$\Vert \theta \Vert_{2} \le \frac{1}{20\sqrt{l}} \quad and \quad \Vert U\theta \Vert_{2} \ge 8r\sqrt{n}.$$
Combing above cases, we show that any $\theta \in \mathbb{R}^{l}$ with $\Vert \theta \Vert_{2} \le 
\frac{1}{20\sqrt{l}}$ satisfies 
$$\mathrm{d}_{A}\left( U\theta,\mathbb{Z}^{n}\right) \ge L \sqrt{\log_{+}{\frac{\alpha\Vert U\theta \Vert_{2}}{L}}}.$$
Then $$\mathrm{RD}_{L,\alpha}^{A}\left( U^{T}\right) \ge \frac{1}{20\sqrt{l}}.$$ 
\par 
Using the (3) of Lemma \ref{lm3.5}, 

$$\mathrm{det}^{1/2}(U^{T}U) \ge 4^{-l}\prod_{j=1}^{l}{\Vert u_{j}\Vert_{2}} \ge 8^{-l}\prod_{j=1}^{l}{d_{j}}.$$ 
Then set $B_{i} = \mathrm{Row}_{i}(B)^{T}$, applying the Proposition \ref{pro2.6} for any $t \ge \delta \sqrt{n} \ge 20l \ge \frac{\sqrt{l}}{\mathrm{RD}_{L,\alpha}^{A}(U^{T})}$, We have  
\begin{align}
\textsf{P}\left( \Vert U^{T}B_{i} \Vert_{2} \le t\sqrt{l} \right)
\le \frac{\left( CL/\alpha \sqrt{l} \right)^{l}}{\mathrm{det}^{1/2}\left( U^{T}U \right)}\left( t+\frac{\sqrt{l}}{\mathrm{RD}_{L,\alpha}^{A}\left( U^{T} \right)} \right)^{l} \le \frac{C^{l}}{\prod_{j=1}^{l}{d_{j}}}t^{l}.\nonumber
\end{align}
Denote that 
$$Y_{i} = \frac{1}{l}\Vert U^{T}B_{i} \Vert_{2}^{2}, \quad M=\frac{C^{2}}{\left( \prod_{j=1}^{l}{d_{j}}\right)^{2/l}}.$$
Then we have 
\begin{align}
\textsf{P}\left( Y_{i} \le s \right) \le \left( Ms \right)^{l/2} \quad \text{for} \ s\ge s_{0}=\delta^{2} n.\nonumber
\end{align}
Applying Lemma \ref{lmn2.1} with $m = l/2$ and $t=4K^{2}s_{0}$, this yields 
\begin{equation}
\begin{aligned}
\textsf{P}\left( \Vert Bu_{j} \Vert_{2} \le 2K\delta n \ for \ all \ j\in [l] \right)
& \le \textsf{P}\left( \sum_{j=1}^{l}{\Vert Bu_{j} \Vert_{2}^{2}} \le 4K^{2}\delta^{2}ln^{2} \right) \\
& = \textsf{P}\left( \sum_{i=1}^{n-k}{Y_{i}} \le n\cdot 4K^{2}s_{0} \right)\\  
& \le \left( \tilde{C}\delta \right)^{l(n-k)} \left( \prod_{j=1}^{l}{\frac{\sqrt{n}}{d_{j}}} \right)^{n-k}.\nonumber
\end{aligned}
\end{equation}
Then 
\begin{equation}
\begin{aligned}
P_{d}
& = \textsf{P}\left( \exists (u_{1},\dots,u_{l})\in \tilde{\mathcal{N}}_{d} :\Vert Bu_{j} \Vert_{2} \le 2K\delta n,\ j\in [l] \right)\\
& \le \left| \tilde{\mathcal{N}}_{d} \right| \cdot \left( \tilde{C} \delta \right)^{l(n-k)} \cdot \left( \prod_{j=1}^{l}{\frac{\sqrt{n}}{d_{j}}} \right)^{n-k}\\
& \le n^{l} \cdot \left( \frac{C\rho^{c}}{r\delta} \right)^{ln}\left( \prod_{j=1}^{l}{\frac{d_{j}}{\sqrt{n}}} \right)^{n} \cdot \left( \tilde{C}\delta \right)^{l(n-k)} \cdot \left( \prod_{j=1}^{l}{\frac{\sqrt{n}}{d_{j}}} \right)^{n}\\
& \le \left( \frac{C_{1}\rho^{c}}{r} \right)^{ln} \cdot \delta^{-lk} \cdot \left( \prod_{j=1}^{l}{\frac{d_{j}}{\sqrt{n}}} \right)^{k}\\
& \le \left( \frac{C_{1}\rho^{c}}{r} \right)^{ln} \cdot \left( \frac{R}{\delta \sqrt{n}} \right)^{lk}\\
& < \left( \frac{\tilde{c}\rho^{c}}{r} \exp{(c\rho^{2})} \right)^{ln}\\
& < \exp{(-2ln)}.\nonumber
\end{aligned}
\end{equation}
where $\rho^{c} < a\rho$ for a sufficiently small constant $a > 0$.
\par 
We now applying Lemma \ref{lm3.4} yields 
\begin{equation}
\begin{aligned}
\textsf{P}\left( W_{d} \cap \ker{(B)} \ne \emptyset \right) 
& \le \textsf{P}\left( W_{d} \cap \ker{(B)} | \Vert B \Vert_{\mathrm{HS}} \le 2Kn \right)+\textsf{P}\left( \Vert B \Vert_{\mathrm{HS}} \ge 2Kn \right)\\
& \le P_{d} + \exp{(-cn^{2})}\\
& \le \exp{(-2ln)}.\nonumber
\end{aligned}
\end{equation}

\end{proof}
Sum up, we can prove Proposition \ref{Propo_1} by applying Lemma \ref{lm3.6} for every $d \in [r\sqrt{n},R]^{l}$.
\begin{proof}[\textsf{Proof of Proposition \ref{Propo_1}}]
Let $\mathcal{E}_{d}$ be the event that $W_{d} \cap \ker{(B)} \ne \emptyset$. Then $\mathcal{E}_{3.1} = \cup{\mathcal{E}_{d}}$.
\par Note that $d_{j}=2^{s_{j}}$ and $s_{j} \in \mathbb{N}\cap [r\sqrt{n},R]$, by estimating the cardinality of $d$, we have 
\begin{align}
\textsf{P}\left( \mathcal{E}_{d} \right) \le \left[\log{\left( \frac{R}{r\sqrt{n}} \right)}\right]^{l}\exp{(-2ln)} \le \exp{(-ln)}.\nonumber
\end{align}

\end{proof}
\section{The proof of Theorem \ref{Theo_main1}}
In this section, we will prove Theorem \ref{Theo_main1} using the inequality of Lemma \ref{lm2.9} and Proposition \ref{Propo_1}. This method relies on Lemma \ref{lm2.13}. Firstly, we need to show that if the RLCD of the subspace of $\ker{(B)}$ is large enough, we can get the large exponential bound.
\begin{mylem} \label{lm4.1}
Let $A$ be a $n \times n$ random matrix with independent entries that satisfy \eqref{Eq_condition}. For $k < c \sqrt{n}$, let $J \subset [n]$ with $|J| = n-k$ and define 
$\mathcal{E}_{J}^{k}$ as the event that exists a linear subspace $E \subset \left( \mathrm{span}\left( \mathrm{Col}_{i}(A),\ for \ all \ i \in J \right) \right)^{\perp}$ such that $\dim{E} \ge k/2$ and 
\begin{align}
\mathrm{RD}_{L,\alpha}^{A}\left( E \right) \ge \exp{(C\frac{n}{k})}.\nonumber
\end{align}
Then 
\begin{equation}
\begin{aligned}
\textsf{P}\left( \mathrm{Col}_{j}(A) \in \mathrm{span}\left( \mathrm{Col}_{i}(A),i\in J \right) ~\text{for}~\text{all}~ \ j \in [n] ~\text{and}~ \ \mathcal{E}_{J}^{k} \right) \le \exp{(-\tilde{c}_{\ref{lm4.1}}nk)},\nonumber
\end{aligned}
\end{equation}
where $\tilde{c}_{\ref{lm4.1}}$ is a constant.
\end{mylem}
In fact, this lemma  can be easily obtained by Lemma 6.1 in Rudelson\cite{Rudelsonrank}.

\begin{proof}[\textsf{Proof of Theorem \ref{Theo_main1}}]

Assume that $\mathrm{rank}(A) \le n-k$. There exists $J \subset [n]$ with $|J|=n-k$ such that $\mathrm{Col}_{j}(A) \in \mathrm{span}\left( \mathrm{Col}_{i}(A),\ i \in J \right)$ for all $j \in [n]$.  
\\
Then 
\begin{equation}
\begin{aligned}
& \textsf{P}\left( \mathrm{rank}(A) \le n-k \right)\\
& = \textsf{P}\left( \exists J \subset [n],|J|=n-k:\mathrm{Col}_{j}(A) \in \mathrm{span}\left( \mathrm{Col}_{i}(A),i \in J  \right), \ for \ all \ j \in [n] \right)\\
& \le \sum_{J \subset [n],|J| = n-k}{\textsf{P}\left( \mathrm{Col}_{j}(A) \in \mathrm{span}\left( \mathrm{Col}_{i}(A),i \in J  \right), \ for \ all \ j \in [n]\setminus J \right)}\\
& \le \binom{n}{k} \sup_{J \subset [n],|J|=n-k}{\textsf{P}\left( \mathrm{Col}_{j}(A) \in \mathrm{span}\left( \mathrm{Col}_{i}(A),i \in J  \right), \ for \ all \ j \in [n]\setminus J \right)}.\nonumber
\end{aligned}
\end{equation}
Note that 
$$\binom{n}{k} \le \exp{\left( k\log{\left( \frac{en}{k} \right)}\right)} < \exp{(ckn)} ,$$
it shows that we only need to prove 
$$\textsf{P}\left( \mathrm{Col}_{j}(A) \in \mathrm{span}\left( \mathrm{Col}_{i}(A),i \in J  \right), \ for \ all \ j \in [n] \right) \le \exp{(-ckn)}.$$
Next we show that $\mathcal{E}_{J}^{k}$ occurs with at least $1-\exp{(-ckn)}$.
\par 
Let $B_{J}$ be an $(n-k) \times n$ random matrix with rows $\mathrm{Row}_{j}(B) = \mathrm{Col}_{n_{j}}(A)^{T}$, where $n_{j}$ is the $j$-th entries of $J$. And let $E_{0} = \ker{(B_{J})}$ and denote by $P_{E_{0}}$ the orthogonal projection onto $E_{0}$.
Let $\tau$ be chosen in Lemma \ref{lm2.9}, and denote $W_{0}=\mathrm{Comp}\left( \tau^{2},\tau^{4} \right)$.
\par 
Applying Lemma \ref{lm2.13} with $E_{0}$, $W_{0}$ and $l = k/4$ yields that at least one of the events described (1) and (2) of lemma occurs. Denote these events $\mathcal{E}_{\ref{lm2.13}}^{(1)}$ and $\mathcal{E}_{\ref{lm2.13}}^{(2)}$ respectively. Using Lemma \ref{lm2.9},
\begin{align}
\textsf{P}\left( \mathcal{E}_{\ref{lm2.13}}^{(1)} \right) \le \exp{(-\frac{k}{4}n)}.
\end{align}
Assume that $\mathcal{E}_{\ref{lm2.13}}^{(2)}$ occurs, there exists $F \subset E_{0}$ with $\dim{F}=\frac{3}{4}k $ such that $F \cap W_{0} = \emptyset$. Let $\rho$ be chosen in Proposition \ref{Propo_1}, and set
$$W_{1}=\left\{ v \in F : 2r\sqrt{n} \le \Vert v \Vert_{2} \le \exp{\left( \frac{\rho^{2}n}{4L^{2}} \right)} \ and \ \mathrm{d}_{A}\left( v,\mathbb{Z}^{n} \right) \le \rho \sqrt{n} \right\}.$$
Applying Lemma\ref{lm2.13}  with $F$, $W_{1}$ and $l$ yields that one of the (1) and (2) of Lemma \ref{lm2.13}. Denote these events by $\mathcal{V}_{\ref{lm2.13}}^{(1)}$ and $\mathcal{V}_{\ref{lm2.13}}^{(2)}$ respectively. Using Proposition \ref{Propo_1},
\begin{align}
\textsf{P}\left( \mathcal{V}_{\ref{lm2.13}}^{(1)} \right) \le \exp{(-\frac{k}{4}n)}.
\end{align}
Assume that $\mathcal{V}_{\ref{lm2.13}}^{(2)}$ occurs. It means that there exist subspace $\bar{F} \subset F$ with $\dim{\bar{F}} = k/2$ such that $\bar{F}\cap W_{1} = \emptyset$. We will show that if this event occurs,
$$\mathrm{RD}_{L,\alpha}^{A}\left( \bar{F} \right) \ge R:=\exp{\left( \frac{\rho^{2}n}{4L^{2}} \right)}.$$
Set $U: \mathbb{R}^{\frac{k}{2}} \to \mathbb{R}^{n}$ such that $U\mathbb{R}^{\frac{k}{2}} = \bar{F}$. Then $\mathrm{RD}_{L,\alpha}^{A}\left( U^{T} \right) = \mathrm{RD}_{L,\alpha}^{A}\left( \bar{F} \right)$. Let $ \theta \in \mathbb{R}^{k/2}$ be a vectors such that 
$$\mathrm{d}_{A}\left( U\theta,\mathbb{Z}^{n} \right) < L\sqrt{\log_{+}{\left( \frac{\alpha \Vert \theta \Vert_{2}}{L} \right)}}.$$
Since $U\mathbb{R}^{k/2}\cap S^{n-1} \subset \mathrm{Incomp}\left( \tau^{2},\tau^{4} \right)$. Applying Lemma \ref{lm2.4} with $\alpha < \alpha_{0}(\tau^{2},\tau^{4},p,K)$ yields,  if $\Vert \theta \Vert_{2} \le 2r \sqrt{n} \le h(\tau^{2},\tau^{4},p,K) \sqrt{n}$, 
$$\mathrm{d}_{A}\left( U\theta,\mathbb{Z}^{n} \right) \ge L\sqrt{\log_{+}{\left( \frac{\alpha\Vert \theta \Vert_{2}}{L} \right)}}.$$
Then $\Vert \theta \Vert_{2} \ge 2r\sqrt{n}$. Note that $U\theta \notin W_{1}$, we have $\Vert \theta \Vert_{2} \ge R$, now we show that $\mathcal{E}_{J}^{k}$ occurs with at least $1-\exp{(-ckn)}$.

Then 
\begin{equation}
\begin{aligned}
&\textsf{P}\left( \mathrm{Col}_{j}(A) \in \mathrm{span}\left( \mathrm{Col}_{i}(A),i \in J  \right), \ for \ all \ j \in [n] \right) \\
& \le 2 \exp{(-\frac{k}{4}n)} + \textsf{P}\left( \mathrm{Col}_{j}(A) \in \mathrm{span}\left( \mathrm{Col}_{i}(A),i \in J  \right), \ for \ all \ j \in [n]  \ and \ \mathcal{V}_{\ref{lm2.13}}^{(2)} \right)\\
& \le \exp{(-\tilde{c}kn)}.\nonumber
\end{aligned}
\end{equation}

\end{proof}
\section{The Proof of Theorem 1.2}
\begin{proof}[\textsf{Proof of Theorem \ref{Theo_main2}}]

Firstly, we prove that inequality is held when $n$ is sufficiently large.

By the Courant-Firsher-Weyl min-max principle: 
$$s_{n-k+1}\left( A \right) = \min_{\mathrm{dim}\left( H \right) = k}{\max_{x\in H \cap S^{n-1}}{\Vert Ax \Vert_{2}}}.$$

This means if $s_{n-k+1}\left( A \right) \le \frac{\varepsilon}{\sqrt{n}}$, then there exists $k$ orthogonal unit vectors $z_{1},\dots,z_{k}$ such that 
\begin{align}
    \Vert Az_{i} \Vert_{2} \le \frac{\varepsilon}{\sqrt{n}}, \quad 1\le i \le k.\nonumber
\end{align}
Let $Z^\top = \left( z_{1},\dots,z_{k} \right)$ be an $n \times k$ full-rank matrix. Applying Lemma \ref{lm2.13} yields that there exists $i_{1},\dots,i_{l} \in [n]$ such that     
\begin{equation}
    \begin{aligned}
        s_{l}\left( Z_{i_{1},\dots,i_{l}} \right)^{-1} 
        & \le C\min_{r \in \{ l+1,\dots,k \}}{\sqrt{\frac{rn}{\left( r-l \right)\sum_{i=r}^{k}{s_{i}(Z)^{2}}}}}\\
        & \le C\min_{r\in\{ l+1,\dots,k \}}{\sqrt{\frac{rn}{(r-l)(k-r+1)}}}\\
        & \le C_{1}\sqrt{\frac{kn}{(k-l)^{2}}}.\nonumber
    \end{aligned}
\end{equation}
Set $F = (Z_{i_{1},\dots,i_{l}})^\top$, $\tilde{F}=(Z_{i_{l+1},\dots,i_{n}})^\top$, and $A_{i} = \mathrm{Col}_{i}\left(  A\right)$. Define the matrix
$A_{t} :=(A_{i_{1}},\dots, A_{i_{l}})$, and $A_{s} = (A_{i_{l+1}},\dots,A_{i_{n}})$.
\\ 
Then 
$$M := AZ^\top = (A_{i_{1}},\dots, A_{i_{l}})F + (A_{i_{l+1}},\dots,A_{i_{n}})\tilde{F} = A_{t}F+A_{s}\tilde{F} .$$
Let $F_{1} = F^\top(FF^\top)^{-1}$ be the right inverse of $A$. Thus, we have 
$$MF_{1} := (A_{i_{1}},\dots,A_{i_{l}}) + (A_{i_{l+1}},\dots,A_{i_{n}})\tilde{F}F_{1}=A_{1} +A_{2}\tilde{F}F_{1}$$
Let $H$ be the linear space spanned by $A_{i_{l+1}},\dots,A_{i_{n}}$. Let $P$ be the orthogonal projection in $\mathbb{R}^{n}$ onto $H^\top=\ker{\left( A_{2}^\top \right)}$. Then
$$PMF_{1} = PA_{1},$$
which yields that 
\begin{align}
    \mathrm{dist}^{2}\left( A_{i_{1}},H \right) + \cdots + \mathrm{dist}^{2}\left( A_{i_{l}},H \right) = \Vert PMF_{1} \Vert_{\mathrm{HS}}^{2} \le \Vert MF_{1} \Vert_{\mathrm{HS}}^{2} \le \Vert F_{1} \Vert_{2}^{2} \Vert M \Vert_{\mathrm{HS}}^{2}.\nonumber
\end{align}
On the one hand, for any $x \in \mathbb{R}^{n}$
$$1 = \Vert FF_{1} \Vert \ge \frac{\Vert FF_{1}x \Vert_{2}}{\Vert x\Vert_{2}} \ge s_{l}\left( F \right) \frac{\Vert F_{1}x \Vert_{2}}{\Vert x\Vert_{2}}. $$
It imply that
\begin{align}
    s_{1}(F_{1}) \le s_{l}(F)^{-1} \le C_{1}\sqrt{\frac{kn}{(k-l)^{2}}}.
\end{align}
On the other hand  
\begin{align}
    \Vert M \Vert_{\mathrm{HS}}^{2} \le \sum_{i=1}^{k}{\Vert \mathrm{Col}_{i}\left( M \right) \Vert_{2}^{2}} = \sum_{i=1}^{k}{\Vert Az_{i} \Vert_{2}^{2}} \le \frac{k\varepsilon^{2}}{\sqrt{n}}.
\end{align}
Applying (5.1) and (5.2) 
\begin{align}
    \Vert MF_{1} \Vert_{\mathrm{HS}} \le \Vert F_{1} \Vert_{2} \Vert M \Vert_{\mathrm{HS}} \le C_{1}\sqrt{\frac{kn}{(k-l)^{2}}} \sqrt{\frac{k\varepsilon^{2}}{n}} \le C_{1}\frac{k\varepsilon}{k-l}. \nonumber
\end{align}
Furthermore, there exists $i_{1},\cdots,i_{l} \in [n]$ satisfying that:
\begin{align}
     \mathrm{dist}^{2}\left( A_{i_{1}},H \right) + \cdots + \mathrm{dist}^{2}\left( A_{i_{l}},H \right) \le \left(  \frac{C_{1}k\varepsilon}{k-l}\right)^{2}.
\end{align}
Now, we have 

 \begin{align}
    \textsf{P}\left( s_{n-k+1}\left( A \right) \le \frac{\varepsilon}{\sqrt{n}} \right)
    & \le \textsf{P}\left( \exists \{ i_{1}, \cdots ,i_{l} \} \subset [n] \  ; \sum_{m=1}^{l}{\Vert P_{H^{\perp}}A_{i_{m}} \Vert_{2}^{2} }  \le \left(  \frac{C_{1}k\varepsilon}{k-l}\right)^{2} \right)\\
    & \le \sum_{i_{1},\dots,i_{l} \in [n]}{\textsf{P}\left( \sum_{m=1}^{l}{\Vert P_{H^{\perp}}A_{i_{m}} \Vert_{2}^{2} } \le \left( \frac{C_{1}k\varepsilon}{k-l} \right)^{2} \right)}.\nonumber
\end{align}

Similar to the proof of Theorem \ref{Theo_main1},  
\begin{align}
\textsf{P}\left( \mathcal{H}:= \exists E \subset H^{\perp}, \mathrm{dim}\left( E \right) = \frac{l}{2} , \mathrm{RD}_{\tilde{L},\alpha}^{A}\left( E \right) \ge R := \exp{\left( \frac{c\rho^{2}n}{l} \right)} \right) \ge 1- 2e^{-cln},\nonumber
\end{align}
where $\tilde{L} = c\sqrt{l}$.
\par 
Then 
\begin{equation}
    \begin{aligned}
        & \textsf{P}\left( \sum_{m=1}^{l}{\Vert P_{H^{\perp}}A_{i_{m}} \Vert_{2}^{2}} \le \left( \frac{C_{1}k\varepsilon}{k-l} \right)^{2} \right) \\
        & \le \textsf{P}\left( \sum_{m=1}^{l}{\Vert P_{H^{\perp}}A_{i_{m}} \Vert_{2}^{2}} \le \left( \frac{C_{1}k\varepsilon}{k-l} \right)^{2} , \mathcal{H} \right) + 2e^{-cln}\\
        & \le \textsf{P}\left( \sum_{m=1}^{l}{\Vert P_{E}A_{i_{m}} \Vert_{2}^{2}} \le \left( \frac{C_{1}k\varepsilon}{k-l} \right)^{2},\mathcal{H} \right) + 2e^{-cln}.\nonumber
    \end{aligned}
\end{equation}
Corollary \ref{col2.8} yields that 
$$\textsf{P}\left( \Vert P_{E}A_{i_{m}} \Vert_{2} \le t \sqrt{l/2} , \mathcal{H}\right) \le (ct)^{l/2} + e^{-cn}.$$
Applying tensorization similar to Lemma \ref{lmn2.1} yields that 
\begin{align}
    \textsf{P}\left( \sum_{m=1}^{l}{\Vert P_{E}A_{i_{m}} \Vert_{2}^{2}} \le l^{2}t^{2},\mathcal{H} \right) \le \left( Ct \right)^{l^{2}/2} + e^{-cln}.\nonumber
\end{align}
It implies that 
\begin{align}
    \textsf{P}\left( \sum_{m=1}^{l}{\Vert P_{E}A_{i_{m}} \Vert_{2}^{2}} \le \left( \frac{Ck\varepsilon}{k-l} \right)^{2},\mathcal{H} \right) \le \left( \frac{Ck\varepsilon}{l(k-l)} \right)^{l^{2}/2} + e^{-cln}.\nonumber 
\end{align}
Furthermore, we have 
\begin{equation}
    \begin{aligned}
        \textsf{P}\left( s_{n-k+1}\left( A \right) \le \frac{\varepsilon}{\sqrt{n}}  \right)
        & \le \binom{n}{l}\left(  \left( \frac{Ck\varepsilon}{l(k-l)} \right)^{l^{2}/2} + \exp{(-cln)} \right)\\
        & \le n^{l} \left( \left( \frac{Ck\varepsilon}{l(k-l)} \right)^{l^{2}/2} + \exp{(-cln)}\right).\nonumber 
    \end{aligned}
\end{equation}
Let $l=\sqrt{2\gamma}k$, $\gamma \in (0,1/2)$. Then 
\begin{align}
    \textsf{P}\left( s_{n-k+1}\left( A \right) \le \frac{\varepsilon}{\sqrt{n}} \right) \le \left( \frac{C\varepsilon}{k} \right)^{\gamma k^{2}} + \exp{(-ckn)}.\nonumber
\end{align} 
Then Theorem \ref{Theo_main2} valid for any $n > n_{0}$. Note that 
\begin{align}
    \lim_{\varepsilon \to 0}{\textsf{P}\left( s_{n-k+1}\left( A \right) \le \frac{\varepsilon}{\sqrt{n}} \right)} = \textsf{P}\left( s_{n-k+1}\left( A \right) = 0 \right) \le e^{-ckn}.\nonumber
\end{align}
where the last inequality is due to Theorem \ref{Theo_main1}. Hence, for any fixed n, there exists an $\varepsilon_{0}\left( n,n_{0} \right)$ such that for any $\varepsilon < \varepsilon_{0}\left( n,n_{0} \right)$ 
\begin{align}
    \textsf{P}\left( s_{n-k+1}\left( A \right) \le \frac{\varepsilon}{\sqrt{n}} \right)-\textsf{P}\left( s_{n-k+1}\left( A \right) = 0 \right) \le e^{-cn_{0}^{2}}.
\end{align}
Set $\varepsilon_{0} = \min{\left\{ \varepsilon(1,n_{0}),\dots,\varepsilon(n_{0}-1,n_{0}) \right\}}$. Applying (5.4) yield that for any $n < n_{0}$ and $\varepsilon < \varepsilon_{0}$ 
$$\textsf{P}\left( s_{n-k+1}(A) \le \frac{\varepsilon}{\sqrt{n}} \right) \le 2 e^{-ckn}.$$
And for any $n < n_{0}$, $\varepsilon > \varepsilon_{0}$, we have 
$$\textsf{P}\left( s_{n-k+1}(A) \le \frac{\varepsilon}{\sqrt{n}} \right) \le 1 \le \left( \frac{C\varepsilon}{k} \right)^{\gamma k^{2}}.$$
where $C= c\sqrt{n_{0}}/\varepsilon_{0}$. Hence Theorem \ref{Theo_main2} valid for any $n$ and $\varepsilon \ge 0$.
\end{proof}



\textbf{Acknowledgment:} Song was partly supported by  the Youth Student Fundamental Research Funds of Shandong University   (No. SDU-QM-B202407). Wang was partly supported by Shandong Provincial Natural Science Foundation (No. ZR2024MA082) and the National Natural Science Foundation of China (No. 12071257).



\begin{thebibliography}{9}

        \bibitem{bour}
        Bourgain, J., Vu, V., Wood, P. (2010).
        On the singularity probability of discrete random matrices.
        \textit{J. Funct. Anal.}
        \textbf{258} 559--603.

        \bibitem{Dai_arXiv}
        Dai, G., Su, Z., Wang, H. (2024).
        Quantitative estimates of the singular values of random i.i.d. matrices.
        \textit{Perprint arXiv}:2412.18912.

         \bibitem{MFV}
    Fernandez, M. (2024).
    A distance theorem for inhomogenous random rectangular matrices.
    \textit{Perprint arXiv}
     :2408.06309.
	
	\bibitem{JAMS_singular}
	Kahn, J., Koml\'{o}s, J. and Szemer\'{e}di, E.  (1995).
	On the probability that a random $\pm 1$-matrix is singular.
	\textit{J. Amer. Math. Soc.}
	\textbf{8} 223--240.

	 \bibitem{Livshyts_1}
	 Livshyts, G. V. (2021).
	 The smallest singular value of heavy-tailed not necessarily i.i.d. random matrices via random rounding.
	 \textit{J. Anal. Math.}
	 \textbf{145} 257--306.
	 
	 \bibitem{Livshyts_2}
	 Livshyts, G. V., Tikhomirov, K. and Vershinin, R. (2021).
	 The smallest singular value of inhomogeneous square random matrices.
	 \textit{Ann. Probab.}
	 \textbf{49} 1286--1309.
     
      \bibitem{Naor_JTDM}
       Naor, A., Youssef, P. (2017).
       Restricted invertibility revisited.
       In: A Journey Through Discrete Mathematics, pp. 657-691. Springer, Cham.
		
	\bibitem{Nguyen_JFA}
	Nguyen, H. (2018).
	Random matrices: overcrowding estimates for the spectrum.
	\textit{J. Funct. Anal.}
	\textbf{275} 2197--2224.
	
	\bibitem{Rebrova_IJM}
	Rebrova, E., Tikhomirov, K. (2018).
	Coverings of random ellipsoids, and invertibility of matrices with i.i.d. heavy-tailed entries.
	\textit{Israel J. of Math.}
	\textbf{227} 507--544.

   
	\bibitem{Rudelsonadvance}
	Rudelson, M., Vershynin, R. (2008).
	The Littlewood-Offord problem and invertibility of random matrices.
	\textit{Adv. Math.}
	\textbf{218} 600--633.
	
	\bibitem{Rudelson}
	Rudelson, M., Vershynin, R. (2009).
	Smallest Singular Value of a Random Rectangular Matrix.
	\textit{Commun. Pure Appl. Math.}
	\textbf{62} 1707--1739.
        
        \bibitem{RudelsonSmall}
        Rudelson, M., Vershynin, R. (2016).
        No-gaps delocalization for general random matrices. 
        \textit{Geom. Funct. Anal.}
        \textbf{26} 1716--1776.
        
	\bibitem{Rudelsonrank}
	Rudelson, M. (2024).
	A large deviation inequality for the rank of a random matrix
	\textit{Ann. Probab.}
	\textbf{52} 1992--2018.
    
	\bibitem{Sconjecture}
        Spielman, D., Teng, S. (2002).
        Smoothed analysis of algorithms. In Proceedings of the International Congress of Mathematicians, Vol. I (Beijing, 2002), pages 597-606.
	
	\bibitem{Tao_RSA}
	Tao, T. and Vu, V. (2006).
	On random $\pm 1$ matrices: singularity and determinant.
	\textit{Random Struct. Algorithms}
	\textbf{28} 1--23.
	
	
	\bibitem{Tao_JAMS}
	Tao, T. and Vu, V. (2007).
	On the singularity probability of random Bernoulli matrices.
	\textit{J. Amer. Math. Soc.}
	\textbf{20} 603--628.
  
     \bibitem{Tikhomirov}
     Tikhomirov, K. (2020).
     Singularity of random Bernoulli matrices.
     \textit{Ann. of Math. (2)}
     \textbf{191} 593--634.

	
	
	
	
\end{thebibliography}
\end{document}